\documentclass[12pt]{nuthesis}	

\usepackage{amsfonts,amssymb, amscd, latexsym, graphicx, color, psfrag, setspace} 
\usepackage[all]{xy}

\newtheorem{dummy}{dummy}[section]
\newtheorem{lemma}[dummy]{Lemma}

\newtheorem{proposition}[dummy]{Proposition}
\theoremstyle{definition}
\newtheorem{definition}[dummy]{Definition}

\author{Oleg Podkopaev}

\title{The Equivalence of Grayson and Friedlander-Suslin Spectral Sequences }

\graduationmonth{June}         

\graduationyear{2012}          

\begin{document}
%
%


\frontmatter		

\maketitle		

\copyrightpage		

\abstract		

This thesis establishes the equivalence of Grayson (\cite{Gr95}) and Friedlander-Suslin (\cite{FS}) spectral sequences, that was previously only
known (\cite{Su01}) for the respective $E_2$-terms.  We develop the necessary techniques regarding $K_0$-presheaves of spectra, building on the 
work of M. Walker (\cite{Wa97}) and construct certain filtrations on the $K$-theory presheaf of spectra that we use as intermediate steps in 
obtaining the equivalence of the filtrations used in respectively \cite{Gr95} and \cite{FS}.

\acknowledgements	

I would like to thank my advisor, Andrei Suslin, for posing the problem and for his guidance, 
support and generosity throughout my period of doctoral study.
I would also like to thank Paul Goerss and Boris Tsygan who agreed to be on my defense comittee.  
At last, I would like to thank the faculty, staff and fellow 
students who made my stay at Northwestern really enjoyable.

\preface		
In \cite{Gr95} Grayson defined certain complexes that (after "globalizing"
them to obtain complexes of sheaves $\mathbb{Z}^{Gr}(i)$) seemed to
be a plausible candidate for "motivic complexes" of Beilinson and
Lichtenbaum.  The main reason behind that was the spectral sequence

$$E^{pq}_2 = \mathbb{H}^{p-q}(X, \mathbb{Z}^{Gr}(-q)) \Rightarrow K_{-p-q}(X)$$

for a smooth scheme $X$ of finite type over a field.

Independently, Suslin and Voevodsky in a series of papers (\cite{FrVo}, \cite{SV96}, \cite{SV98},
\cite{MVW}) defined complexes $\mathbb{Z}(i)$ that turned out to satisfy
most of Beilinson-Lichtenbaum conjectures.  In particular, using
Bloch and Lichtenbaum's work \cite{BL} as a starting
point, Friedlander and Suslin (\cite{FS}) constructed a
spectral sequence

$$E^{pq}_2 = \mathbb{H}^{p-q}(X, \mathbb{Z}(-q)) \Rightarrow K_{-p-q}(X)$$

for a smooth scheme $X$ of finite type over a field and proved
several of its properties.

Later, Suslin (\cite{Su01}) showed that Grayson's motivic complexes
$\mathbb{Z}(i)$ are quasiisomorphic to Suslin-Voevodsky motivic
complexes $\mathbb{Z}(i)$ and, consequently, the $E_2$-terms of the
respective spectral sequences are isomorphic.  This result left open
the question whether the spectral sequences themselves are
isomorphic.

In the present paper we answer this question affirmatively.  The
importance of our result is in that it shows that
\cite{FS} does not depend on the main result of
unpublished work \cite{BL}.

In order to establish the isomorphism of the above spectral
sequences we compare the filtrations

$$\mathcal{K}(X \times \Delta^{\ast})=\mathcal{W}^0(X \times \Delta^{\ast}) \leftarrow \mathcal{W}^1(X \times \Delta^{\ast}) \leftarrow \dots$$

and

$$\mathcal{K}(X \times \Delta^{\ast})=\mathcal{K}^0(X \times \Delta^{\ast}) \leftarrow \mathcal{K}^1(X \times \Delta^{\ast}) \leftarrow \dots$$

considered in, respectively, \cite{Gr95} and \cite{FS}.
Namely, we consider the presheaf of simplicial spectra on the big
Nisnevich site of the base field given by

$$U \rightarrow \mathcal{K}(U \times \Delta^{\ast}),$$

and introduce the "presheafified" versions of the filtrations of
\cite{Gr95} and \cite{FS}.  Our main result is that there is
a morphism of filtrations $$\mathcal{W}^{\ast} \rightarrow
\mathcal{K}^{\ast}$$ that is a termwise weak equivalence of
presheaves of simplicial spectra\footnote{that is, a stalkwise weak
equivalence, see \cite{Jar87} for the corresponding definitions
concerning presheaves of spectra.}.  The isomorphism of the spectral
sequences easily follows.

We would like now to describe the contents of the paper by section.
The purpose of Sections 1.1 -1.3 is to assemble the notions used in our definitions and
arguments and define the filtrations we want to compare.  For reader's convenience we give a brief account of
Jardine's homotopy theory of (bi)simplicial and presheaves of
spectra on a Grothendieck site and extend his definitions to the
case of presheaves of simplicial spectra.  We also recall the main
definitions concerning Waldhausen categories \cite{Wa85} and
multidimensional cubes of Waldhausen categories.  Finally, we extend
M. Walker's \cite{Wa97} definition of abelian $K_0$-presheaves to the
case of presheaves of spectra and prove a version of an important lemma of Voevodsky
(Cor. 5.11 \cite{Vo99}) in our setting.

In Section 1.4 we construct the filtrations we want to
compare.  In order to do this we need to define several $K$-theory
spectra\footnote{that are modifications of Walker's bivariant
$K$-theory} and prove some of their properties.  In order to
establish weak equivalences between several $K$-theory spectra we
show that the natural functors between the underlying categories of
complexes of sheaves induce equivalences of derived categories, and
use Theorem 1.9.8 of \cite{TT88}.

Chapter 2 contains the proof of our main result.  We define
several intermediate filtrations on the presheaf $\mathbf{K}(\ast
\times \Delta^{\ast})$ and construct the weak equivalences in
question as compositions of weak equivalences of consecutive
intermediate filtrations\footnote{We note that the first
intermediate filtration is a version of the one considered in
\cite{Wa97}}.

\tableofcontents	


\mainmatter             



\chapter{Preliminaries}	

\section{Homotopy theory of presheaves of spectra and presheaves of simplicial spectra}
\subsection{Presheaves of spectra}
By a spectrum $X$ we mean a sequence $(X^n, x_n)$ of pointed
simplicial sets together with pointed morphisms $\sigma^n:\Sigma X^n
\rightarrow X^{n+1}$ of simplicial sets.  A morphism $X \rightarrow
Y$ of spectra is a sequence of pointed morphisms $X^n \rightarrow
Y^n$ of simplicial sets, that respect the morphisms $\sigma^n$. The
stable homotopy groups of a spectrum $X$ are defined as $\pi_i^s (X)
= \text{colim} \pi_{i+n}(X^n, x_n)$.  In the present work we will be
using the stable model structure of \cite{BF} on the category of spectra,
in particular by a weak equivalence of spectra we mean a stable weak
equivalence, i.e. a morphism that induces isomorphisms on the stable
homotopy groups.

Let $\mathcal{C}$ be a Grothendieck site with enough points.

\begin{definition}A presheaf of spectra on $\mathcal{C}$ is a contravariant
functor from $\mathcal{C}$ to the category of spectra.
\end{definition}
We recall here the model structure on the category of presheaves of
spectra on $\mathcal{C}$, introduced in \cite{Jar97}. First define the
$i$-th sheaf of homotopy groups of a presheaf $\mathcal{F}$ of
spectra as the abelian sheaf $\pi^s_i\mathcal{F}$ on $\mathcal {C}$
associated with the presheaf:

$$U \rightarrow \pi^s_i(\mathcal{F}(U))$$

Now, a  cofibration is a morphism $\mathcal{F} \rightarrow
\mathcal{F}^{\prime}$ that is a sectionwise cofibration of spectra
in the sense of \cite{BF}.

A weak equivalence is a "stable weak equivalence", that is, a
morphism $\mathcal{F} \rightarrow \mathcal{F}^{\prime}$ such that
the induced morphisms of abelian sheaves: $$\pi_i^s\mathcal{F}
\rightarrow \pi_i^s\mathcal{F}^{\prime}$$ are isomorphisms of
sheaves for all $i$.

The fibrations are the morphisms that have the right lifting
property with respect to all the morphisms that are cofibrations and
weak equivalences.

\subsection{Presheaves of simplicial spectra} We first fix
certain definitions definitions regarding simplicial spectra (see
[Jardine]).  By a simplicial spectrum we mean a sequence $X =
(X^n),$ $n\geq 0$ of pointed bisimplicial sets together with
morphisms of bisimplicial sets
$$\sigma^n: X^n \wedge S^1 \rightarrow X^{n+1},$$

where $S^1$ is the standard simplicial circle regarded as a
bisimplicial set constant in one variable.

With a simplicial spectrum $X$ we associate the following two
spectra:

\begin{itemize}

\item the diagonal $DX$ defined via $(DX)_n = D(X^n)\footnote{It is indeed a spectrum because there is a canonical isomorphism:

$$D(X^n \wedge S^1) \simeq D(X^n) \wedge S^1.$$}$

\item the $n$-th vertical spectrum defined via $v_nX = ((v_nX)^k =
X^k_{n,\ast}).$

\end{itemize}

The second construction gives rise to the series of simplicial
abelian groups
$$(\pi_{i}^{vert}X)_{\ast} = \pi_i^s(v_{\ast}X),$$

that are involved in the "simplicial spectra version" of
Bousfield-Friedlander spectral sequence, due to Jardine (\cite{Jar97}):

\begin{proposition} (Jardine):  Let $X$ be a simplicial spectrum. Then there is a
spectral sequence:

$$E_2^{p,q} = \pi_p((\pi_{i}^{vert}X)_{\ast}) \Rightarrow \pi_{p+q}DX$$

\end{proposition}

We end this section with the following two definitions:

\begin{definition}A presheaf of simplicial spectra on a site $\mathcal{C}$ is a
contravariant functor from $\mathcal{C}$ to the category of
simplicial spectra.
\end{definition}

\begin{definition}A morphism $\mathcal{F} \rightarrow
\mathcal{F}^{\prime}$ of presheaves of simplicial spectra is a weak
equivalence of presheaves of simplicial spectra if it induces a weak
equivalence of respective diagonal presheaves.
\end{definition}

\subsection{The "singular complex" construction}

Let $\mathcal{S}m(k)$ be the site of schemes over $k$ with Nisnevich
(or Zariski) topology.

Let $\Delta^{\ast}$ be the standard cosimplicial scheme.  If
$\mathcal{F}$ is a presheaf of spectra, we denote
$C_{\ast}\mathcal{F}$ the following presheaf of simplicial spectra:

$$C_{\ast}\mathcal{F}(U) = \mathcal{F}(U \times \Delta^{\ast}).$$

We now prove a useful criterion of weak contractibility of
$C_{\ast}\mathcal{F}$.

\begin{lemma}Let $\mathcal{F}$ be a presheaf of spectra. Suppose that for
every scheme $U/k$ there is a morphism $H_U: \mathcal{F}(U)
\rightarrow \mathcal{F}(U \times \mathbb{A}^1)$ satisfying the
following properties:

(1)  $H$ is natural in $U$, meaning that if $V \rightarrow U$ is a
morphism of $k$-schemes, then the square

$$
  \begin{CD}
      \mathcal{F}(V)          @>>>          \mathcal{F}(V \times \mathbb{A}^1)     \\
   @AAA                        @AAA  \\
      \mathcal{F}(U)          @>>>          \mathcal{F}(U \times \mathbb{A}^1),
  \end{CD}
  $$

strictly commutes.

(2) $i_1^{\ast} \circ H_U = id_{\mathcal{F}(U)}$

(3) $i_0^{\ast} \circ H_U$ is homotopic to the constant map to the
distinguished point of $\mathcal{F}(U)$ via a natural in $U$
homotopy.

Then the presheaf of simplicial spectra $C_{\ast}\mathcal{F}$ is
weakly contractible.
\end{lemma}
\begin{proof} We need to prove that $C_{\ast}\mathcal{F}$ is contractible
stalkwise, but we show that it is even contractible sectionwise. Let
$U$ be a $k$-scheme.
  Since $H$ satisfies the
property (1), it induces a morphism of simplicial spectra:
$$(C_{\ast}\mathcal{F})(U) \rightarrow (C_{\ast}\mathcal{F})(U \times \mathbb{A}^1).$$

Now, by \cite{FS}, Lemma 7.1, the morphisms
$C_{\ast}i_1^{\ast}$ and $C_{\ast}i_0^{\ast}$ are homotopic.  So the
identity of $(C_{\ast}\mathcal{F})(U)$ is homotopic to a map that is
in turn homotopic to the constant map.
\end{proof}

\section{Waldhausen categories}

\subsection{Waldhausen $K$-theory spectra} Recall that a pointed category is a category with a
distinguished zero object.  The following definition is due to
Waldhausen:

\begin{definition}A Waldhausen category is a pointed category $\mathcal{C}$
together with two subcategories $co\mathcal{C}$ (the cofibrations)
and $w\mathcal{C}$ (the weak equivalences) that satisfy respectively
the following two sets of axioms:

Cof 1.  The isomorphisms in $\mathcal{C}$ are in $co\mathcal{C}$.

Cof 2.  For every $A \in \mathcal{C}$, the arrow $\ast \rightarrow
A$ is in $co\mathcal{C}$.

Cof 3.  Cofibrations admit cobase changes:  if $A\rightarrowtail B$
is a cofibration, and $A\rightarrow C$ any arrow in $\mathcal{C}$,
then 1) the pushout $C\cup_A B$ exist and 2) the canonical arrow $C
\rightarrow C\cup_A B$ is a cofibration.

Weq 1.  The isomorphisms in $\mathcal{C}$ are in $w\mathcal{C}$.

Weq 2.  (Gluing axiom).  If in a commutative diagram

%

$$
\xymatrix{
 B  \ar[d] & A \ar[l] \ar[r] \ar[d] & C \ar[d] \\
 B'   & A' \ar[l] \ar[r]  & C'
}
$$

the horizontal arrows on the left are cofibrations and all three
certical arrows are in $w\mathcal{C}$, then the induced map

$$B\cup_A C \rightarrow B^{\prime}\cup_{A^{\prime}} C^{\prime}$$ is
in $w\mathcal{C}$.
\end{definition}

Note: Following Waldhausen we will denote $B/A$ any representative of $\ast \cup_A B$

\begin{definition}A sequence in $\mathcal{C}$ is called a cofibration sequence if it is equivalent
to a sequence of the form:
$$A \rightarrowtail B \twoheadrightarrow B/A.$$
\end{definition}

\begin{definition}
Let $\mathcal{A}$ and $\mathcal{B}$ be Waldhausen categories and $F:
\mathcal{A} \rightarrow \mathcal{B}$ a functor. we say that $F$ is
exact if it takes the zero object of $\mathcal{A}$ to the zero
object of $\mathcal{B}$, cofibrations in $\mathcal{A}$ to
cofibrations in $\mathcal{B}$, and weak equivalences in
$\mathcal{A}$ to weak equivalences in $\mathcal{B}$.
\end{definition}
Let $\mathcal{C}$ be an exact category.  Following Waldhausen, we
define the $K$-theory spectrum $\mathcal{K}(\mathcal{C})$ as
follows.  First, we define a simplicial exact category
$\mathcal{S}_{\ast}\mathcal{C}$.  We set
$\mathcal{S}_{n}\mathcal{C}$ to be the category whose objects are
$n$-diagrams of cofibrations:

$$A_1 \rightarrowtail A_2 \rightarrowtail \dots \rightarrowtail A_n$$

with a choice of quotients $A_{i,j}= A_j/A_i.$  Morphisms are just
morphisms of such diagrams.  This category can itself be given a
structure of a Waldhausen category.

Now consider the following sequence of simplicial sets:
$$\mathcal{K}(\mathcal{C})^0 =
\Omega|w\mathcal{S}_{\ast}\mathcal{C}|, \mathcal{K}(\mathcal{C})^1 =
|w\mathcal{S}_{\ast}\mathcal{C}|, \dots, \mathcal{K}(\mathcal{C})^n
= |w\mathcal{S}_{\ast}^{(n)}\mathcal{C}|, \dots
$$

It is proved in \cite{Wa85} that this sequence is an
$\Omega$-spectrum.

\begin{definition}
The $K$-theory spectrum of the Waldhausen category $\mathcal{C}$ is
the sequence $\mathcal{K}(\mathcal{C})^{\ast}$ with the standard
structure morphisms.
\end{definition}

We now define "relative" $K$-theory spectrum for an exact functor
$F: \mathcal{A} \rightarrow \mathcal{B}$ between two exact
categories.  We start with defining the simplicial exact category
$\mathcal{S}_{\ast}(\mathcal{A} \rightarrow \mathcal{B})$.  By
definition it is the pullback of the diagram:

$$\mathcal{S}_{\ast}\mathcal{A}\rightarrow \mathcal{S}_{\ast}\mathcal{B} \leftarrow
P\mathcal{S}_{\ast}\mathcal{B},$$ where
$P\mathcal{S}_{\ast}\mathcal{B}$ is the categorical path space of
$\mathcal{S}_{\ast}\mathcal{B}$.

Explicitly, the objects of $\mathcal{S}_n(F: \mathcal{A} \rightarrow
\mathcal{B})$ are pairs of diagrams of cofibrations:

$$A_{0,1} \rightarrowtail \dots \rightarrowtail A_{0,n},$$ $A_{0,i} \in \mathcal{A}$ and

$$B_0 \rightarrowtail B_1 \rightarrowtail \dots B_n,$$ $B_i \in
\mathcal{B}$, together with an isomorphism of diagrams:

$$F(A_{0,1}) \rightarrowtail \dots \rightarrowtail F(A_{0,n}) \cong B_1/B_0 \rightarrowtail
\dots B_n/B_0.$$

Now we define the relative $K$-theory spectrum to be the following sequence of simplicial
sets:
$$
\mathcal{K}(F:\mathcal{A} \rightarrow
\mathcal{B})^0 = \Omega |w\mathcal{S}_{\ast}\mathcal{S}_{\ast}(F:
\mathcal{A} \rightarrow \mathcal{B})|, 
$$
$$
\mathcal{K}(F:\mathcal{A} \rightarrow
\mathcal{B})^1 = |w\mathcal{S}_{\ast}\mathcal{S}_{\ast}(F:
\mathcal{A} \rightarrow \mathcal{B})|,
$$
$$\vdots$$
$$\mathcal{K}(F:\mathcal{A} \rightarrow
\mathcal{B})^n = |w\mathcal{S}_{\ast}\mathcal{S}_{\ast}^{(n)}(F:
\mathcal{A} \rightarrow \mathcal{B})|,$$
$$\vdots$$

\begin{definition}
The relative $K$-theory spectrum $\mathcal{K}(F:\mathcal{A}
\rightarrow \mathcal{B})$ of an exact functor $F:\mathcal{A}
\rightarrow \mathcal{B}$ is the sequence $\mathcal{K}(F:\mathcal{A}
\rightarrow \mathcal{B})^{\ast}$ with the standard structure
morphisms.
\end{definition}
The motivation for this construction is the following proposition
due to Waldhausen:

\begin{proposition}  Let $F: \mathcal{A} \rightarrow \mathcal{B}$ be an exact
functor between exact categories.  Then the sequence:

$$|w\mathcal{S}_{\ast}\mathcal{A}| \rightarrow |w\mathcal{S}_{\ast}\mathcal{A}| \rightarrow
|w\mathcal{S}_{\ast}\mathcal{S}_{\ast}(F: \mathcal{A} \rightarrow
\mathcal{B})|$$ has the homotopy type of a fibration sequence.
\end{proposition}
We notice that it follows from the proposition above that the
following sequence of spectra:

$$\mathcal{K}(\mathcal{A}) \rightarrow \mathcal{K}(\mathcal{A}) \rightarrow
\mathcal{K}(F: \mathcal{A} \rightarrow \mathcal{B})$$ has the
homotopy type of a fibration sequence (or, a cofibration sequence,
since for spectra these two notions coincide).

\begin{lemma} Let

$$
\xymatrix{
 \mathcal{B}'  \ar[r]^{J} & \mathcal{B}     \\
 \mathcal{A}' \ar[u]^{F'} \ar[r]^{I}   & \mathcal{A} \ar[u]_{F} ,
}
$$


be a diagram of exact categories and exact functors, strictly
commutative on objects.  Then the vertical functors $F^{\prime}$ and
$F$ induce a morphism

$$\mathcal{K}(I:\mathcal{A}^\prime \rightarrow
\mathcal{A}) \rightarrow \mathcal{K}(J:\mathcal{B}^\prime
\rightarrow \mathcal{B}).$$

\end{lemma}
\begin{proof} It is clearly sufficient to construct a simplicial exact
functor:

$$\mathcal{S}_{\ast}(I: \mathcal{A}^{\prime} \rightarrow
\mathcal{A}) \rightarrow \mathcal{S}_{\ast}(J: \mathcal{B}^{\prime}
\rightarrow \mathcal{B}).$$

Now, the objects of $\mathcal{S}_{n}(\mathcal{A}^{\prime}
\rightarrow \mathcal{A})$ are pairs of diagrams of cofibrations:

$$A_{0,1}^{\prime} \rightarrowtail \dots \rightarrowtail A_{0,n}^{\prime},$$
$A_{0,i}^{\prime} \in \mathcal{A}^{\prime}$ and

$$A_0 \rightarrowtail A_1 \rightarrowtail \dots A_n,$$ $A_i \in
\mathcal{A}$, together with an isomorphism of diagrams:

$$I(A_{0,1}^{\prime}) \rightarrowtail \dots \rightarrowtail I(A_{0,n}^{\prime}) \cong A_1/A_0 \rightarrowtail
\dots A_n/A_0$$  and

the objects of $\mathcal{S}_{n}(\mathcal{B}^{\prime} \rightarrow
\mathcal{B})$ are pairs of diagrams of cofibrations:

$$B_{0,1}^{\prime} \rightarrowtail \dots \rightarrowtail B_{0,n}^{\prime},$$
$B_{0,i}^{\prime} \in \mathcal{B}^{\prime}$ and

$$B_0 \rightarrowtail B_1 \rightarrowtail \dots B_n,$$ $B_i \in
\mathcal{B}$, together with an isomorphism of diagrams:

$$J(B_{0,1}^{\prime}) \rightarrowtail \dots \rightarrowtail J(B_{0,n}^{\prime}) \cong B_1/B_0 \rightarrowtail
\dots B_n/B_0.$$

Consider the sequence of exact functors $$F_{n}:
\mathcal{S}_{n}(\mathcal{A}^{\prime} \rightarrow \mathcal{A})
\rightarrow \mathcal{S}_{n}(\mathcal{B}^{\prime} \rightarrow
\mathcal{B}) $$ given by the rule that sends the pair of
cofibrations $$A_{0,1}^{\prime} \rightarrowtail \dots
\rightarrowtail A_{0,n}^{\prime},$$
$$A_0 \rightarrowtail A_1 \rightarrowtail \dots A_n$$

to the pair $$F^{\prime}(A_{0,1}^{\prime}) \rightarrowtail \dots
\rightarrowtail F^{\prime}(A_{0,n}^{\prime}),$$
$$F(A_0) \rightarrowtail F(A_1) \rightarrowtail \dots F(A_n).$$

Now, we have isomorphisms $$JF^{\prime}(A_{0,i}^{\prime}) =
FI(A_{0,i}^{\prime}) \simeq F(A_i/A_0) = F(A_i)/F(A_0),$$

so the pair $$F^{\prime}(A_{0,1}^{\prime}) \rightarrowtail \dots
\rightarrowtail F^{\prime}(A_{0,n}^{\prime}),$$
$$F(A_0) \rightarrowtail F(A_1) \rightarrowtail \dots F(A_n)$$
represents an element in $\mathcal{S}_n(J: \mathcal{B}^{\prime}
\rightarrow \mathcal{B})$.

Moreover, it is clear that these functors commute with the face and
degeneracy operators, so we can consider them as a simplicial
functor, that is the functor we were constructing.
\end{proof}

\subsection{Cubical diagrams of Waldhausen categories}

Here we extend the constuctions of the previous subsection to
$n$-dimensional cubes of exact categories.  We start with the
general definitions concerning $n$-cubes in categories.

Let $[\mathbf{n}]$ be a finite set with $n$ elements.  Denote
$\mathcal{P}(n)$ the category of its subsets.  An $n$-cube in a
category $\mathcal{C}$ is a functor $\mathcal{P}(n) \rightarrow
\mathcal{C}$.  More explicitly an $n$-cube in $\mathcal{C}$ is
given by a collection of objects $X_{i_0 i_1 \dots i_{n-1}}, (i_0i_1
\dots i_{n-1}) \in \{0,1\}^{\times n}$ and arrows $d_k = : X_{i_0 i_1
\dots 1 \dots i_{n-1}} \rightarrow X_{i_0i_1 \dots 0 \dots i_{n-1}}$
such that for $k < l$ all diagrams:

$$
  \begin{CD}
      X_{i_0i_1 \dots 1 \dots 0 \dots i_{n-1}}          @>>>          X_{i_0i_1 \dots 0 \dots 0 \dots i_{n-1}}     \\
   @AAA                        @AAA  \\
      X_{i_0i_1 \dots 1 \dots 1 \dots i_{n-1}}          @>>>          X_{i_0i_1 \dots 0 \dots 1 \dots i_{n-1}},
  \end{CD}
  $$

commute.

A morphism of $n$ cubes $X \rightarrow X^{\prime}$ is just a natural
transformation of functors, or a collection of arrows $$X_{i_0i_1
\dots i_{n-1}} \rightarrow X^{\prime}_{i_0i_1 \dots i_{n-1}}$$ such
that all the relevant diagrams commute.  Note that an $n$-cube in
$\mathcal{C}$ can be considered as a morphism of two $n-1$-cubes in
$n$ different ways.

We will be interested in two instances of the above construction:

\begin{itemize}

\item $\mathcal{C}$ is the category of Waldhausen categories and exact
functors

\item $\mathcal{C}$ is the category of (simplicial) spectra and
morphisms of (simplicial) spectra

\end{itemize}

Namely, consider the following situation.  We are given an $n$-cube
$$\mathcal{X}= \mathcal{X}_{i_0i_1 \dots i_{n-1}}, F= F_{i_0i_1
\dots i_{n-1}}$$ of Waldhausen categories and exact functors.
Following Grayson, we are going to associate to it a certain
$n$-simplicial category, whose $n$-fold diagonal we will call "the
iterated cofiber of the cube $\mathcal{X}$".  The construction is
$n$-dimensional generalization of what we called "relative
$K$-theory" construction of the previous section.

We start with introducing the following notation:  if $A$ is an
ordered set regarded as a category we denote the arrow $i
\rightarrow j, i\leq j$ by $j/i$.  If $A$ is an ordered set and
$\mathcal{C}$ an exact category we call a functor $Ar{A} \rightarrow
\mathcal{C}$ exact, if $F(i/i) = \ast$ for all $i$ and for all $i
\leq j \leq k$ the sequence $$F(j/i) \rightarrow F(k/i) \rightarrow
F(k/j)$$ is a cofibration sequence.

Going back to the cube $\mathcal{X}= \mathcal{X}_{i_0i_1 \dots
i_{n-1}}$, $F= F_{i_0i_1 \dots i_{n-1}}$, we define its iterated
cofiber as $n$-simplicial category $C\mathcal{X},$ given by:

$$C\mathcal{X}(A_1, A_2, \dots, A_n) =$$
$$= Exact([Ar(A_1) \rightarrow \{L\}Ar(A_1)]\times [Ar(A_2)
\rightarrow \{L\}Ar(A_2)] \times \dots \times [Ar(A_n) \rightarrow
\{L\}Ar(A_n)], \mathcal{X}),$$

where $L$ is a one element ordered set.

It is clear that a morphism of $n$-cubes of exact categories induces
an $n$-simplicial functor on iterated cofibers.

\section{$K_0$-presheaves of spectra}
Throughout the rest of the text  by a field we mean a field over which resolution of singularities holds and by a scheme we mean a smooth
scheme of finite type over such a field.

The category $K_0(Sm(k))$ is defined in \cite{Wa97}.

\begin{definition}(\cite{Wa97}) An abelian  $K_0$-presheaf on $Sm(k)$ is a contravariant
functor $K_0(Sm(k)) \rightarrow Ab$.
\end{definition}
\begin{definition} $K_0$-presheaf of spectra is a presheaf $\mathcal{F}:Sm(k)
\rightarrow Spectra$ with a $K_0$-presheaf structure on each
presheaf $\pi_i^s \mathcal{F}$.
\end{definition}
We denote $F_{Nis}$ the sheafification of $F$ in Nisnevich topology, and $h_i(F)$ the $i$-th cohomology of the complex associated to the simplicial sheaf $C_*(F)$.
\begin{lemma}  Let $F$ be an abelian $K_0$-presheaf such that $F_{Nis}=0$.
Then $(h_i(F))_{Nis}=0$ for $i\geq 0$.
\end{lemma}
\begin{proof} By \cite{Wa97} Lemma 7.5 the following conditions
are equivalent for an abelian $K_0$-presheaf:

(1) $(h_i(F))_{Nis}=0$ for $i\geq 0$

(2) $Ext^i(F_{Nis}, G)=0$ for all $i \geq 0$ and every $G$ that is a
homotopy invariant pretheory which is also a Nisnevich sheaf.

Since in our case $F_{Nis}=0$, the condition (2) is satisfied, so we
have $(h_i(F))_{Nis}=0$ for $i\geq 0$.

\end{proof}
\begin{lemma}  Let $\mathcal{F}$ be a weakly contractible $K_0$-presheaf of
spectra. Then the simplicial presheaf of spectra
$C_{\ast}\mathcal{F}$ is weakly contractible.
\end{lemma}
\begin{proof} Since $\mathcal{F}$ is weakly contractible, by the very
definition we have for $i \geq 0$
$$(\pi_i\mathcal{F})_{Nis}=0.$$  So by Lemma 1 $i \geq
0$

$$(C_{\ast}\pi_i\mathcal{F})_{Nis}$$ is a contractible simplicial
abelian sheaf.  So for every henselian local scheme $X$ over $k$,
the simplicial abelian group $$\pi_i\mathcal{F}(X) \leftarrow
\pi_i\mathcal{F}(X \times \Delta^1) \leftarrow \pi_i\mathcal{F}(X
\times \Delta^2) \leftarrow \dots$$ is contractible, or,for all $j
\geq 0$,

$$\pi_j(\pi_i\mathcal{F}(X
\times \Delta^{\ast})) =0$$

Now, the simplicial abelian group above can be interpreted as
$\pi_i^{vert} C_{\ast}\mathcal{F}(X)$.

Consider now the Bousfield-Friedlander-Jardine spectral sequence of
Proposition 2.2:

$$\pi_j(\pi_i C_{\ast}\mathcal{F}(X)) \Rightarrow \pi_{i+j}D(C_{\ast}\mathcal{F}(X)).$$

By the previous discussion the spectral sequence degenerates proving
that $C_{\ast}\mathcal{F}(X)$ is contractible for every henselian
local scheme, or that $C_{\ast}\mathcal{F}$ is a contractible
presheaf.

\end{proof}

\section{Filtrations on the $K$-theory presheaf of spectra of a scheme}

\subsection{Presheaves of $K$-theory spectra}

\subsubsection{Categories of complexes of finitely supported sheaves}

\begin{definition} Let $X$ be a noetherian scheme.  Let $Sch/X$ denote its big
Zariski site.  A big coherent sheaf $M$ on $X$ is a contravariant functor $Sch/X
\rightarrow Ab$ such that:

(1)  for every $U \in Sch/X$ the restriction $M_U$ of $M$ to $Sch/U$
is a coherent $\mathcal{O}_U$-module and

(2)  for every arrow $f:U \rightarrow V$ in $Sch/X$ the induced
homomorphism $f^{\ast}M_V \rightarrow M_U$ is an isomorphism.
\end{definition}

\begin{definition} A big vector bundle $E$ on $X$ is a contravariant functor $Sch/X
\rightarrow Ab$ such that:

(1)  for every $U \in Sch/X$ the restriction $E_U$ of $E$ to $Sch/U$
is a locally free coherent $\mathcal{O}_U$-module and

(2)  for every arrow $f:U \rightarrow V$ in $Sch/X$ the induced
homomorphism $f^{\ast}E_V \rightarrow E_U$ is an isomorphism.
\end{definition}

With these definitions at our disposal we introduce the following
category that is central to our considerations.  Let $S$ and $X$ be
schemes over $k$.

\begin{definition}
The category $\mathcal{P}(S,X)$ is the full subcategory of the
category of bounded complexes in the category of big coherent
sheaves over $S \times X$, consisting of complexes $M^{\ast}$ such
that for every $i$, $M^i$ is supported on a finite union of
subschemes of $S \times X$ finite over $S$ and $(p_{S})_{\ast}M^{i}$
is a locally free coherent $\mathcal{O}_S$-module.
\end{definition}
It is straightforward to establish the following:
\begin{proposition}
The category $\mathcal{P}(S,X)$ has a structure of a Waldhausen
category with cofibrations - degreewise split monomorphisms and weak
equivalences - quasiisomorphisms.
\end{proposition}

The following two propositions are due to M. Walker (\cite{Wa97}): 
\begin{proposition}
Let $S \rightarrow S^{\prime}$ and $X \rightarrow X^{\prime}$ be a
morphism of schemes over $k$. Then the inverse image of sheaves and
the direct image of sheaves induce morphisms of spectra
$\mathcal{K}(\mathcal{P}(S^{\prime}, X)) \rightarrow
\mathcal{K}(\mathcal{P}(S, X))$ and $\mathcal{K}(\mathcal{P}(S, X))
\rightarrow \mathcal{K}(\mathcal{P}(S, X^{\prime}))$.
\end{proposition}

\begin{proposition} (\cite{Wa97}) Let $U$ and $V$ be schemes. The following diagram is homotopy
cartesian for every henselian local scheme $X$:

$$
  \begin{CD}
      \mathcal{K}(\mathcal{P}(X \times \Delta^{\ast},V))          @>>>          \mathcal{K}(\mathcal{P}(X \times \Delta^{\ast},U\cup V))     \\
   @AAA                        @AAA  \\
      \mathcal{K}(\mathcal{P}((X \times \Delta^{\ast},U\cap V))          @>>>          \mathcal{K}(\mathcal{P}(X \times \Delta^{\ast},U)),
  \end{CD}
  $$

\end{proposition}

The following purity result is easy: 
\begin{proposition}
Let $i: Z \rightarrow X$ be an inclusion of a closed
subscheme.  Denote $\mathcal{P}^{Z}(S, X)$ the full subcategory of
$\mathcal{P}(S, X)$ consisting of complexes $M^{\ast}$ with every
$M^{i}$ supported on a subscheme of $S \times Z$.  Then the direct image
induces 
a weak equivalence the
respective $K$-theory spectra.
\end{proposition}

\begin{proposition} (Homotopy invariance in the second argument): Let
 $X$ be a scheme over $k$.
 The natural morphism of presheaves $$C_{\ast}\mathcal{K}(\mathcal{P}(\ast, X)) \rightarrow C_{\ast}\mathcal{K}(\mathcal{P}(\ast, X \times
 \mathbb{A}^1))$$ induced by embedding $X \rightarrow X \times \mathbb{A}^1$ given by $$x \rightarrow (x,0)$$
  is a weak equivalence of presheaves.
\end{proposition}

\begin{proof}  We prove that the equivalence holds sectionwise, which
implies the weak equivalence of the stalks.

So, let $U$ be a scheme over $k$.  We need to prove that the
morphism of simplicial spectra $$\mathcal{K}(\mathcal{P}(U \times
\Delta^{\ast}, X)) \rightarrow \mathcal{K}(\mathcal{P}(U\times
\Delta^{\ast}, X \times
 \mathbb{A}^1))$$
is a weak equivalence of simplicial spectra.  First, it is
sufficient to show that the morphism of bisimplicial sets is a weak equivalence:

$$|w\mathcal{S}_{\ast}(U \times \Delta^{\ast},
X)| \rightarrow |w\mathcal{S}_{\ast}(U\times \Delta^{\ast}, X \times
 \mathbb{A}^1)|$$

Consider the fiber sequence of simplicial sets:

$$|w\mathcal{S}_{\ast}\mathcal{P}(U \times \Delta^{\ast},
X)| \rightarrow |w\mathcal{S}_{\ast}\mathcal{P}(U\times
\Delta^{\ast}, X \times
 \mathbb{A^1})| \rightarrow |w\mathcal{S}_{\ast}\mathcal{S}_{\ast}(\mathcal{P}(U, X) \rightarrow
 \mathcal{P}(U, X \times \mathbb{A}^1))|$$

Denote the presheaf $$U \rightarrow
|w\mathcal{S}_{\ast}\mathcal{S}_{\ast}(\mathcal{P}(U, X) \rightarrow
 \mathcal{P}(U, X \times \mathbb{A}^1))|$$ by $\mathcal{F}$.  We need to
show that $C_{\ast}\mathcal{F}$ is sectionwise contractible.  By
Lemma 2.5 it is sufficient to constuct a family of morphisms $$H_U:
\mathcal{F}(U) \rightarrow (C_{1}\mathcal{F})(U)=\mathcal{F}(U
\times \mathbb{A}^1)$$ functorial in $U$ such that

(1)  $i_1^{\ast} \circ H_U = id_{\mathcal{F}(U)}$

(2)  $i_0^{\ast} \circ H_U$ is homotopic to the constant map via a
homotopy natural in $U$

Consider the following two fibration sequences of simplicial sets:

$$|w\mathcal{S}_{\ast}\mathcal{P}(U, X)| \rightarrow |w\mathcal{S}_{\ast}\mathcal{P}(U, X \times \mathbb{A}^1)|
\rightarrow |w\mathcal{S}_{\ast}\mathcal{S}_{\ast}(\mathcal{P}(U, X)
\rightarrow
 \mathcal{P}(U, X \times \mathbb{A}^1))|$$

$$|w\mathcal{S}_{\ast}\mathcal{P}(U \times \mathbb{A}^1, X)| \rightarrow |w\mathcal{S}_{\ast}\mathcal{P}(U \times
\mathbb{A}^1, X \times \mathbb{A}^1)| \rightarrow
|w\mathcal{S}_{\ast} \mathcal{S}_{\ast}(\mathcal{P}(U \times
\mathbb{A}^1, X) \rightarrow
 \mathcal{P}(U \times \mathbb{A}^1, X \times \mathbb{A}^1))|$$

We need to construct a vertical arrow between the last terms of
these sequences.  According to Lemma 1, it is sufficient to
construct two functors:

$$\mathcal{P}(U, X) \rightarrow \mathcal{P}(U \times \mathbb{A}^1,
X)$$  and

$$\mathcal{P}(U, X \times \mathbb{A}^1) \rightarrow \mathcal{P}(U \times
\mathbb{A}^1, X \times \mathbb{A}^1)$$

making the diagram:

$$
\xymatrix{
\mathcal{P}(U, X)  \ar[d] \ar[r] & \mathcal{P}(U, X \times \mathbb{A}^1) \ar[d] \\
\mathcal{P}(U \times \mathbb{A}^1, X) \ar[r] & \mathcal{P}(U \times \mathbb{A}^1, X \times \mathbb{A}^1)
}
$$

commutative.

Let $\Delta_{\mathbb{A}^1}$ denote the subscheme of $$U \times
\mathbb{A}^1 \times (X \times \mathbb{A}^1) \times \mathbb{A}^1$$
consisting of the points $(u,t,x,s,t)$ and
$\delta_{\mathbb{A}^1}$ be the subscheme of $$U \times \mathbb{A}^1
\times X \times \mathbb{A}^1$$ consisting of points $(u,t,x,t)$.

We constuct the first vertical arrow as the composition of functors:

$$\mathcal{P}(U, X) \rightarrow \mathcal{P}(U \times \mathbb{A}^1,
X \times \mathbb{A}^1) \rightarrow \mathcal{P}(U \times
\mathbb{A}^1, X),$$ where the functor (1) is given by:

$$E^{\ast} \rightarrow i_{\delta_{\mathbb{A}^1} \ast} i_{\delta_{\mathbb{A}^1}}^{\ast}(E^{\ast} \times \mathcal{O}_{\mathbb{A}^1}
\times \mathcal{O}_{\mathbb{A}^1})$$

and the functor (2) by:

$$E^{\ast} \rightarrow (id_U \times id_{\mathbb{A}^1} \times id_X \times 0)_{\ast}E^{\ast}$$

The second vertical arrow is the composition of functors:

$$\mathcal{P}(U, X \times \mathbb{A}^1) \rightarrow \mathcal{P}(U \times
\mathbb{A}^1 \times \mathbb{A}^1, X \times \mathbb{A}^1) \rightarrow
\mathcal{P}(U \times \mathbb{A}^1, X \times \mathbb{A}^1),$$

where the functor (1) is given by:

$$E^{\ast} \rightarrow i_{\Delta_{\mathbb{A}^1} \ast} i_{\Delta_{\mathbb{A}^1}}^{\ast}(E^{\ast} \times \mathcal{O}_{\mathbb{A}^1}
\times \mathcal{O}_{\mathbb{A}^1})$$

and the functor (2) by:

$$E^{\ast} \rightarrow (id_U \times id_{\mathbb{A}^1} \times id_X \times \mu)_{\ast}E^{\ast},$$

where $$\mu: \mathbb{A}^1 \times \mathbb{A}^1 \rightarrow
\mathbb{A}^1$$

is the multiplication map:

$$(s,t) \rightarrow st.$$

It can be easily seen that these two vertical arrows make the diagram above
commutative.

To prove the claim we need to check that

(1) $i_1^{\ast}(\text{the second vertical arrow}) = id$ and

(2) $i_0^{\ast}(\text{the second vertical arrow})$ factors through
the category $\mathcal{P}(U \times \mathbb{A}^1, X)$

Let's prove the first claim.  We need to compare the stalks of
$$E^{n}$$ and $$i_1^{\ast}(id_U \times id_{\mathbb{A}^1} \times id_X
\times \mu)_{\ast}i_{\Delta_{\mathbb{A}^1} \ast}
i_{\Delta_{\mathbb{A}^1}}^{\ast}(E^{n} \times
\mathcal{O}_{\mathbb{A}^1} \times \mathcal{O}_{\mathbb{A}^1}).$$

The map $(id_U \times id_{\mathbb{A}^1} \times id_X \times \mu)
\circ i_{\Delta_{\mathbb{A}^1}}$ is given by:

$$(u,t,x,s,t) \rightarrow (u,t,x,st),$$ we see that it is an
isomorphism in the neighborhood of $(u,1,x,s,1)$, so the stalks of
the sheaf that we consider are indeed just the stalks of $E^n$.

Now, for the second claim, putting $t=0$ we get

$$(u,0,x,s,0) \rightarrow (u,0,x,0),$$ from where it is clear that
our functor factors through $\mathcal{P}(U \times \mathbb{A}^1, X)$.

\end{proof}

\begin{definition}
The category $\mathcal{M}^{fin}(S,X)$ is the full subcategory of the
category of bounded complexes in the category of big coherent
sheaves over $S \times X$, consisting of complexes $M^{\ast}$ such
that for every $i$, $M^i$ is supported on a finite union of
subschemes of $S \times X$ finite over $S$.
\end{definition}

\begin{proposition}
(\cite{Wa97})  Let $X = C_1 \times C_2 \times \dots \times C_n$, where
$C_i$ are smooth affine curves.  Then the obvious inclusion of
categories induces a weak equivalence of spectra:
$$\mathcal{K}(\mathcal{P}(S, X)) \rightarrow
\mathcal{K}(\mathcal{P}^{fin}(S, X)).$$
\end{proposition}

\subsubsection{Categories of sheaves with quasifinite support}

\begin{definition}
The category $\mathcal{P}^{qf}(S,X)$ is the full subcategory of the
category of bounded complexes in the category of big vector bundles
over $S \times X$, consisting of complexes $E^{\ast}$ such that for
every $i$, $H^i(E)$ is supported on a finite union of subschemes of
$S \times X$ quasifinite over $S$.
\end{definition}

It is straighforward to check:
\begin{proposition}
The category $\mathcal{P}^{qf}(S,X)$ has a structure of a saturated
biWaldhausen category with canonical homotopy pushouts and
canonical homotopy pullbacks.
\end{proposition}

\begin{definition}
The category $\mathcal{M}^{qf}(S,X)$ is the full subcategory of the
category of bounded complexes in the category of coherent sheaves
over $S \times X$, consisting of complexes $M^{\ast}$ such that for
every $i$, $H^i(M^{\ast})$ is supported on a finite union of
subschemes of $S \times X$ quasifinite over $S$.
\end{definition}
\begin{proposition}
The inclusion of categories $\mathcal{P}^{qf}(S,X) \rightarrow
\mathcal{M}^{qf}(S,X)$ induces an equivalence on $K$-theory spectra.
\end{proposition}

\begin{proof}
Since all the schemes in question are regular, the proof follows by \cite{TT88} Theorem 1.9.8 from the equivalence of $w^{-1}\mathcal{P}^{qf}(S,X)$ and $w^{-1}\mathcal{M}^{qf}(S,X)$, that is a well known result of Quillen (\cite{Qu73}).
\end{proof}

\begin{definition}
The category $\mathcal{M}(S,X)$ is the full subcategory of the
category of bounded complexes in the category of coherent sheaves
over $S \times X$, consisting of complexes $M^{\ast}$ such that for
every $i$, $M^i$ is supported on a finite union of subschemes of $S
\times X$ quasifinite over $S$.
\end{definition}

\begin{lemma}
Both categories $\mathcal{M}^{qf}(S,X)$ and $\mathcal{M}(S,X)$ have
the structure of biWaldhausen categories with canonical homotopy
pushouts and canonical homotopy pullbacks.
\end{lemma}

\begin{proof}
The case of $\mathcal{M}(S \times X)$ is obvious.  Indeed,
let $f: M^{\ast} \rightarrow M^{\prime \ast}$ and $g: M^{\ast}
\rightarrow M^{\prime \prime \ast}$ be morphisms of complexes in
$\mathcal{M}(S \times X)$.  By definition of the canonical
homotopy pushout, its terms are given by $$(M^{\prime \ast}
\cup_{M^{\ast}}^{h} M^{\prime \prime \ast})_n = M^{\prime n} \oplus
M^{n+1} \oplus M^{\prime \prime n},$$

so they all are supported on the union of supports of $M^{\prime
\ast}$, $M^{\ast}$, and $M^{\prime \prime \ast}$, which is a union
of a finite number of subschemes of $S \times X$ finite over $S$.

Now let $f: M^{\prime \ast} \rightarrow M^{\ast}$ and $g: M^{\prime
\prime \ast} \rightarrow M^{\ast}$ be morphisms of complexes in
$\mathcal{M}(S \times X)$.  Then the terms of the canonical
homotopy pullback are given by
$$(M^{\prime \ast} \times_{M^{\ast}}^{h} M^{\prime \prime \ast})_n = M^{\prime n} \oplus M^{n-1} \oplus M^{\prime \prime n},$$
and they have supports of the needed type.

We now consider the case of $\mathcal{M}^{qf}(S \times X))$.  Let
$f: M^{\ast} \rightarrow M^{\prime \ast}$ and $g: M^{\ast}
\rightarrow M^{\prime \prime \ast}$ be morphisms of complexes in
$\mathcal{M}^{qf}(S \times X)$.

We have an exact sequence of sheaves:

$$\dots \rightarrow H^{n}(M^{\prime \ast}) \oplus H^{n}(M^{\prime \prime \ast}) \rightarrow H^{n}(M^{\prime \ast} \cup_{M^{\ast}}^{h} M^{\prime \prime \ast}) \rightarrow H^{n+1}(M^{\ast}) \rightarrow \dots.$$

Let $M^{\ast}$ be acyclic outside the union of subschemes $Z_1, Z_2,
\dots Z_{k_1}$,

$M^{\prime \ast}$ be acyclic outside the union of subschemes
$Z_1^{\prime}, Z_2^{\prime}, \dots Z_{k_2}^{\prime}$,

$M^{\prime \prime \ast}$ be acyclic outside the union of subschemes
$Z_1^{\prime \prime}, Z_2^{\prime \prime}, \dots Z_{k_3}^{\prime
\prime}$.  We see from the exact sequence above that
$H^{n}(M^{\prime \ast} \cup_{M^{\ast}}^{h} M^{\prime \prime \ast})$
have zero stalks outside of the union of $$Z_1, Z_2, \dots Z_{k_1},
Z_1^{\prime}, Z_2^{\prime}, \dots Z_{k_2}^{\prime}, Z_1^{\prime
\prime}, Z_2^{\prime \prime}, \dots Z_{k_3}^{\prime \prime},$$ so
the complex $M^{\prime \ast} \cup_{M^{\ast}}^{h} M^{\prime \prime
\ast}$ is in the category $\mathcal{M}^{qf}(S \times X)$, as
desired.

The case of canonical homotopy pushouts is treated similarly using
the dual exact sequence of cohomology sheaves.

\end{proof}

\begin{proposition}
The inclusion of categories $\mathcal{M}(S,X) \rightarrow
\mathcal{M}^{qf}(S,X)$ induces an equivalence on $K$-theory spectra.
\end{proposition}

\begin{proof}
We checked that the categories in question satisfy the conditions of
\cite{TT88} Theorem 1.9.8.  So it is sufficient to show
that $F$ induces an equivalence of homotopy categories:

$$w^{-1}\mathcal{M}(S \times X) \rightarrow \mathcal{M}^{qf}(S \times X)$$

It is clear that the induced functor is fully faithful.  We need to
show that the functor is essentially surjective: that every complex
$M^{\ast}$ in $\mathcal{M}^{qf}(S \times X)$ is isomorphic in
$w^{-1}\mathcal{M}^{qf}(S \times X)$ to a complex
$\mathcal{M}(S \times X)$.

We show that by induction on the lenghth of the complex
$M^{\ast}$
=$$\dots \rightarrow M^{-2} \rightarrow M^{-1} \rightarrow M^0 \rightarrow 0 \rightarrow \dots,$$
where $M^0$ is in located in degree $0$.
  Consider the ``good truncation`` $\tau_{\leq -1}M^{\ast}$.  We have an inclusion of complexes:
$$\tau_{\leq -1}M^{\ast} \rightarrow M^{\ast}.$$  Denote its mapping cone $C^{\ast}$.  It is representented by the factor complex:

$$\dots \rightarrow 0 \rightarrow 0 \rightarrow M^{-1}/ker \delta^{-1} \rightarrow M^{0} \rightarrow 0 \rightarrow \dots,$$ which is quasiisomorphic to the complex: $$\dots \rightarrow 0 \rightarrow 0 \rightarrow M^0/ \delta^{-1} = H^0(M^{\ast}) \rightarrow 0 \rightarrow \dots,$$ concentrated in degree $0$.

By \cite{TT88} the category $w^{-1}\mathcal{M}^{qf}(S \times
X))$ is triangulated.  We have a distinguished triangle in
$w^{-1}\mathcal{M}^{qf}(S \times X)$:

$$\tau_{\leq -1}M^{\ast} \rightarrow M^{\ast} \rightarrow C^{\ast} \rightarrow \tau_{\leq -1}M^{\ast} [1].$$

The first term is isomorphic to an object in
$F(w^{-1}\mathcal{M}(S \times X))$ by inductive hypothesis,
and the same is true for the third term, because $H^0(M^{\ast})$ has
a support on a finite union of subschemes of $S \times X$ that are
quasifinite over $S$.

We conclude that the second term is also in the essential image of
$F$, which ends the proof.
\end{proof}

\subsection{Filtrations}

In this section we introduce the filtrations on the $K$-theory presheaf of spectra that we aim to compare.
\subsubsection{Grayson's filtration}  Our construction is the globalization of the following construction of Grayson (\cite{Gr95}).  
Let $X$ be a scheme.  Consider the closed embedding $X \times \{1\} \rightarrow X \times \mathbb{G}_m$ as a $1$-cube in 
the category of schemes.  Denote its $q$-th power, in the sense of \cite{Gr95}, by $(X \times \{1\} \rightarrow X \times \mathbb{G}_m)^{\times q}$.  By functoriality we obtain a $q$-cube
of Waldhausen categories $\mathcal{P}(X \times \{1\}) \rightarrow \mathcal{P}(X \times \mathbb{G}_m)^{\times q}$.
By 1.0.2.2 we associate to this $q$-cube a $q$-simplicial Waldhausen category, which is its iterated cofiber, and that we denote
  $(\mathcal{P}(X \times \{1\}) \rightarrow \mathcal{P}(X \times \mathbb{G}_m))^{\wedge q}$. Applying the $K$-functor to its $q$-fold
 diagonal we get a spectrum 

$$\mathcal{K}(\mathcal{P}(X \times \{1\}) \rightarrow \mathcal{P}(X \times \mathbb{G}_m)^{\wedge q})$$.

Now, by the functoriality properties of the category $\mathcal{P}(S, X)$ there is a presheaf of simplicial spectra:  
$$C_{\ast}\mathcal{K}(\mathcal{P}(\ast \times \{1\}) \rightarrow \mathcal{P}(\ast \times \mathbb{G}_m)^{\wedge q})$$

Now the $q$-th term of globalized Grayson's filtration is given by the $q$-fold delooping 

$$\mathcal{W}^q = \Omega^{-q}C_{\ast}\mathcal{K}(\mathcal{P}(\ast \times \{1\}) \rightarrow \mathcal{P}(\ast \times \mathbb{G}_m)^{\wedge q}).$$

The maps $\mathcal{W}^q \rightarrow \mathcal{W}^{q-1}$ are defined the same way as in \cite{Gr95}.

\subsubsection{Friedlander-Suslin filtration}  The globalized version of the construction of Friedlander and Suslin (\cite{FS}) is defined as follows.  Consider 
the presheaf $\mathcal{K}^q = C_{\ast}(\mathcal{P}^{qf}(\ast, \mathbb{A}^q))$.  The maps $\mathcal{K}^q \rightarrow \mathcal{K}^{q-1}$ induced 
by closed embeddings $\mathbb{A}^{q-1} \rightarrow \mathbb{A}^{q}$ give us a filtration on $\mathcal{K}$.

\chapter{Comparison of two filtrations}	

In this chapter we will prove that the two filtrations constructed in Section 1.4.2 are equivalent. The proof will consist of several steps.

\subsection{Step 1} We will need the globalized version 
of the filtration constructed by Walker in \cite{Wa97}.

The same way as when constructing Grayson's filtration we take the morphism $\{1\} \rightarrow \mathbb{P}^1$ as a starting point and obtain
presheaves of simplicial spectra:

$$C_{\ast}\mathcal{K}((\mathcal{P}(\ast, \{1\}) \rightarrow
\mathcal{P}(\ast, \mathbb{P}^1))^{\wedge q}).$$ These give us the terms of the filtration.  The
morphisms from  $q$-th term to the $q-1$-st term are constructed as
follows.The $q$-th term is constructed via a $q$-cube of exact categories,
the ${q-1}$-st is constructed via a $q-1$-cube.  We extend the
latter cube to a $q$-cube without changing the iterated cofiber by
adding an extra face consisting of trivial categories.

Now, Let $$i: X \times \mathbb{P}^1 \times \mathbb{P}^1 \times \dots
\times 0 \rightarrow X \times \mathbb{P}^1 \times \mathbb{P}^1
\times \dots \times \mathbb{P}^1$$ be the obvious closed inclusion.
Taking inverse image with respect to $i$ defines a morphism of
$q$-cubes that defines a functor on the iterated cofibers that we need. 

The following Proposition is an easy consequence of \cite{Wa97}.

\begin{proposition}
There is a weak equivalence of presheaves of spectra 
$$
\mathcal{W}^q \cong C_{\ast}\mathcal{K}((\mathcal{P}(\ast, \{1\}) \rightarrow
\mathcal{P}(\ast, \mathbb{P}^1))^{\wedge q}).
$$
\end{proposition}

\subsection{Step 2}

In this section we prove the following

\begin{proposition} The morphism between the two filtrations
$$C_{\ast}\mathcal{K}((\mathcal{P}(\ast, \{1\}) \rightarrow
\mathcal{P}(\ast, \mathbb{P}^1))^{\wedge q})$$ and
$$C_{\ast}\mathcal{K}((\mathcal{P}(\ast, \mathbb{P}^1-0) \rightarrow
\mathcal{P}(\ast, \mathbb{P}^1))^{\wedge q})$$ induced by the obvious functors is a weak equivalence.
\end{proposition}
\begin{proof}:  Let $X$ be a henselian local scheme. We proceed by induction on $q$.

1.  $q$=1.  By homotopy invariance lemma, the vertical arrows in the
diagram:

$$
  \begin{CD}
      |\mathcal{K}(X \times \Delta^{\ast},\mathbb{P}^1-0)|          @>>>          |\mathcal{K}(X \times \Delta^{\ast},\mathbb{P}^1)|     \\
   @AAA                        @AAA  \\
      |\mathcal{K}(X \times \Delta^{\ast},\{1\})|          @>>>          |\mathcal{K}(X \times \Delta^{\ast},\mathbb{P}^1)|,
  \end{CD}
  $$

are weak equivalences.  Therefore the cofibers of the horizontal arrows are weakly equivalent as well. 


2.  Induction, using homotopy invariance in the second argument.  For the sake of simplicity of notation we
treat case $q=2$.  The general case is treated similarly.

Consider the following diagram:

$$
\xymatrix @!=5pc{
  & \mathcal{P}(X \times \Delta^{\ast},\mathbb{P}^1-0 \times \mathbb{P}^1-0) \ar[rr] \ar'[d][dd]
      &  & \mathcal{P}(X \times \Delta^{\ast},\mathbb{P}^1-0 \times \mathbb{P}^1) \ar[dd]        \\
  \mathcal{P}(X \times \Delta^{\ast},\{1\} \times \{1\}) \ar[ur]\ar[rr]\ar[dd]
      &  & \mathcal{P}(X \times \Delta^{\ast},\{1\} \times \mathbb{P}^1) \ar[ur]\ar[dd] \\
  & \mathcal{P}(X \times \Delta^{\ast},\mathbb{P}^1 \times \mathbb{P}^1-0) \ar'[r][rr]
      &  & \mathcal{P}(X \times \Delta^{\ast},\mathbb{P}^1 \times \mathbb{P}^1)                \\
  \mathcal{P}(X \times \Delta^{\ast},\mathbb{P}^1 \times \{1\}) \ar[rr]\ar[ur]
      &  & \mathcal{P}(X \times \Delta^{\ast},\mathbb{P}^1 \times \mathbb{P}^1) \ar[ur]        }
$$

Denote the face

$$
  \begin{CD}
      \mathcal{P}(X \times \Delta^{\ast},\mathbb{P}^1 \times \mathbb{P}^1)          @>>>          \mathcal{P}(X \times \Delta^{\ast},\mathbb{P}^1 \times \mathbb{P}^1)     \\
   @AAA                        @AAA  \\
      \mathcal{P}(X \times \Delta^{\ast},\{1\} \times \mathbb{P}^1)          @>>>           \mathcal{P}(X \times \Delta^{\ast},\mathbb{P}^1-0 \times \mathbb{P}^1),
  \end{CD}
  $$

by $\mathbf{face} \ 1$ and the face

$$
  \begin{CD}
      \mathcal{P}(X \times \Delta^{\ast},\mathbb{P}^1 \times \{1\})          @>>>          \mathcal{P}(X \times \Delta^{\ast},\mathbb{P}^1 \times \mathbb{P}^1-0)     \\
   @AAA                        @AAA  \\
      \mathcal{P}(X \times \Delta^{\ast},\{1\} \times \{1\})          @>>>           \mathcal{P}(X \times \Delta^{\ast},\mathbb{P}^1-0 \times \mathbb{P}^1-0),
  \end{CD}
  $$

  by $\mathbf{face} \ 2$.

  We have cofiber sequence:

$$\mathcal{K}(iter. \ cofib. \mathbf{face} \ 1) \rightarrow \mathcal{K}(iter. \ cofib. \mathbf{face} \ 2)
\rightarrow \mathcal{K}(iter.\  cofib.\ of\  3-\mathbf{cube}).$$

The first two spectra are contractible by homotopy invariance lemma.
Again, since the iterated cofiber does not depend on the choice of
iteration, we obtain the equivalence of $K$-theory of the front and
the rear faces of the 3-cube, which was our objective.

\end{proof}

\subsection{Step 3}

In this section we prove

\begin{proposition}

$$C_{\ast}\mathcal{K}((\mathcal{P}(\ast, \mathbb{P}^1-0) \rightarrow
\mathcal{P}(\ast, \mathbb{P}^1))^{\wedge q})$$ and

$$C_{\ast}\mathcal{K}((\mathcal{P}(\ast, \mathbb{A}^1-0) \rightarrow
\mathcal{P}(\ast, \mathbb{A}^1))^{\wedge q})$$ are weakly equivalent.
\end{proposition}

\begin{proof}

We use induction on $q$.

$q=1$:  By the Mayer-Vietoris the square

$$
  \begin{CD}
      |\mathcal{K}(X \times \Delta^{\ast},\mathbb{P}^1-\infty)|          @>>>          |\mathcal{K}(X \times \Delta^{\ast},\mathbb{P}^1)|     \\
   @AAA                        @AAA  \\
      |\mathcal{K}(X \times \Delta^{\ast},\mathbb{P}^1-\infty - 0)|          @>>>          |\mathcal{K}(X \times \Delta^{\ast},\mathbb{P}^1-0)|,
  \end{CD}
  $$

is homotopy cartesian for every henselian local scheme $X$.  But it
is just another way of writing the square:

$$
  \begin{CD}
      |\mathcal{K}(X \times \Delta^{\ast},\mathbb{A}^1)|          @>>>          |\mathcal{K}(X \times \Delta^{\ast},\mathbb{P}^1)|     \\
   @AAA                        @AAA  \\
      |\mathcal{K}(X \times \Delta^{\ast},\mathbb{A}^1- 0)|          @>>>          |\mathcal{K}(X \times \Delta^{\ast},\mathbb{P}^1-0)|,
  \end{CD}
  $$

$q \geq 2$.  Again, we treat the case $q=2$ for simplicity of notation, the general case is treated similarly.  Consider the diagram
of exact categories and exact functors:

$$\xymatrix @!=5pc{
  &  \mathcal{P}(X \times \Delta^{\ast},\mathbb P^1-0\times \mathbb P^1-0)  \ar[rr] \ar'[d][dd]
      &  &  \mathcal{P}(X \times \Delta^{\ast},\mathbb P^1-0\times \mathbb P^1)  \ar[dd]        \\
   \mathcal{P}(X \times \Delta^{\ast},\mathbb A^1-0\times \mathbb A^1-0)  \ar[ur]\ar[rr]\ar[dd]
      &  &  \mathcal{P}(X \times \Delta^{\ast},\mathbb A^1-0\times \mathbb A^1 )  \ar[ur]\ar[dd] \\
  &  \mathcal{P}(X \times \Delta^{\ast},\mathbb P^1\times \mathbb P^1-0)  \ar'[r][rr]
      &  &  \mathcal{P}(X \times \Delta^{\ast},\mathbb P^1\times \mathbb P^1)                 \\
   \mathcal{P}(X \times \Delta^{\ast},\mathbb A^1\times \mathbb A^1-0) \ar[rr]\ar[ur]
      &  &  \mathcal{P}(X \times \Delta^{\ast},\mathbb A^1\times \mathbb A^1)  \ar[ur]        }
$$

Note that the top and bottom squares are ``Cech'' diagrams, and therefore the iterated cofiber of the cube is trivial. Thus, the cofibers of the front and of the rear squares are weakly equivalent, as we wanted to prove.

\end{proof}

\subsection{Step 4}

As in the previous section, let $X$ be a henselian local scheme.  Here we prove

\begin{proposition} $$C_{\ast}\mathcal{K}((\mathcal{P}(\ast,
\mathbb{A}^1-0) \rightarrow \mathcal{P}(\ast,
\mathbb{A}^1))^{\wedge q})$$ and
$$C_{\ast}\mathcal{K}(\mathcal{P}(\ast, \mathbb{A}^q-0) \rightarrow
\mathcal{P}(\ast, \mathbb{A}^q))$$ are weakly equivalent.

\end{proposition}
\begin{proof}
Again, we use induction on $q$.  The case $q=1$ is tautological,
let's treat the case $q=2$.  Consider the following cube:

$$
\xymatrix @!=5pc{
  & \mathcal{P}(X \times \Delta^{\ast},\mathbb A^1-0\times \mathbb A^1-0) \ar[rr] \ar'[d][dd]
      &  & \mathcal{P}(X \times \Delta^{\ast},\mathbb A^1-0\times \mathbb A^1) \ar[dd]        \\
  \mathcal{P}(X \times \Delta^{\ast},\mathbb A^1-0\times \mathbb{A}^1-0) \ar[ur]\ar[rr]\ar[dd]
      &  & \mathcal{P}(X \times \Delta^{\ast},\mathbb A^1-0\times \mathbb{A}^1) \ar[ur]\ar[dd] \\
  & \mathcal{P}(X \times \Delta^{\ast},\mathbb A^1\times \mathbb A^1-0) \ar'[r][rr]
      &  & \mathcal{P}(X \times \Delta^{\ast},\mathbb A^1\times \mathbb A^1)                \\
  \mathcal{P}(X \times \Delta^{\ast},\mathbb A^1\times \mathbb A^1-0) \ar[rr]\ar[ur]
      &  & \mathcal{P}(X \times \Delta^{\ast},\mathbb A^2-0) \ar[ur]        }
$$

Here the front face is the "Cech" cube of the cover of
$\mathbb{A}^2-0$ by $\mathcal{U}_0= (\mathbb{A}^1-0) \times
\mathbb{A}^1$ and $\mathcal{U}_1=\mathbb{A}^1 \times
(\mathbb{A}^1-0)$.

The $K$-theory of the front face is contractible, so the $K$-theory
of the iterated cofiber is the same as $K$-theory of the iterated
cofiber of the rear face, which is
$$C_{\ast}\mathcal{K}((\mathcal{P}(X, \mathbb{A}^1-0)
\rightarrow \mathcal{P}(X, \mathbb{A}^1))^{\wedge 2}).$$

On the other hand $K$-theory of the top is contractible , so the
$K$-theory of the iterated cofiber of the 3-cube is the same as the
$K$-theory of the bottom, which is
$$C_{\ast}\mathcal{K}(\mathcal{P}(X, \mathbb{A}^2-0)
\rightarrow \mathcal{P}(X, \mathbb{A}^2)),$$ so the two
spectra are equivalent.
\end{proof}
\subsection{Step 5} Let $X$ be a scheme over $k$.  Denote $\mathcal{P}^{qf, X \times (\mathbb{A}^q - 0)}(X,
\mathbb{A}^q)$ the full subcategory of $\mathcal{P}^{qf}(X,
\mathbb{A}^q)$ whose objects are complexes with cohomology supported
on $X\times (\mathbb{A}^q - 0)$.  The goal of this section is
proving 
\begin{proposition}
$$C_{\ast}\mathcal{K}(\mathcal{P}(\ast,
\mathbb{A}^q-0) \rightarrow \mathcal{P}(\ast, \mathbb{A}^q))$$

and $$C_{\ast}\mathcal{K}(\mathcal{P}^{qf, \ast \times (\mathbb{A}^q
- 0)}(\ast, \mathbb{A}^q) \rightarrow \mathcal{P}^{qf}(\ast,
\mathbb{A}^q))$$ are weakly equivalent.
\end{proposition}

\begin{proof}
Note that by the results of section 3 we have weak equivalences:

$$C_{\ast}\mathcal{K}(\mathcal{P}(\ast,
\mathbb{A}^q-0) \rightarrow \mathcal{P}(\ast, \mathbb{A}^q)) \cong
C_{\ast}\mathcal{K}(\mathcal{M}^{fin}(\ast, \mathbb{A}^q-0)
\rightarrow \mathcal{M}^{fin}(\ast, \mathbb{A}^q))
$$

and

$$C_{\ast}\mathcal{K}(\mathcal{P}^{qf, \ast \times (\mathbb{A}^q
- 0)}(\ast, \mathbb{A}^q) \rightarrow \mathcal{P}^{qf}(\ast,
\mathbb{A}^q)) \cong C_{\ast}\mathcal{K}(\mathcal{M}^{qf,\ast \times
(\mathbb{A}^q - 0)}(\ast, \mathbb{A}^q) \rightarrow
\mathcal{M}^{qf}(\ast, \mathbb{A}^q))$$

(Note that we can apply Proposition 3.10 becauuse $\mathbb{A}^q =
\mathbb{A}^1 \times \mathbb{A}^1 \times \dots \times \mathbb{A}^1$
and $\mathbb{A}^q - 0$ can be covered by open sets of the form
$\mathbb{A}^1 \times \mathbb{A}^1 \times \dots  \times
(\mathbb{A}^1-0) \times \dots \times \mathbb{A}^1$ with all
intersections also being products of $\mathbb{A}^1$ and
$\mathbb{A}^1$.)

 Let
$X$ be a henselian local scheme over $k$. The simplicial category
$$S_{\ast}(\mathcal{M}^{qf, X\times (\mathbb{A}^q - 0)}(X,
\mathbb{A}^q) \rightarrow \mathcal{M}^{qf}(X, \mathbb{A}^q))$$ has
as its $n$-objects $(n + 1)$-filtered bounded complexes
$$M^{\ast}_0 \subseteqq M^{\ast}_1 \subseteqq \dots \subseteqq
M^{\ast}_n$$ of coherent sheaves on $X\times \mathbb{A}^q$
satisfying the following conditions:

(1) $M^j$ is supported on a finite union of subschemes of $X \times
\mathbb{A}^q$ quasifinite over $X$

(2) all the factors $M^{j}_i/ M^{j}_0$ are supported on $X\times \mathbb{A}^q - 0$.

Every scheme $Z$ quasifinite over a henselian local scheme $X$ has a
canonical decomposition $Z = Z^+ \coprod Z^-$, where $Z^+$ is finite
over $X$ and the projection of $Z^-$ to $X$ does not contain the
closed point of $X$.  This gives us a direct product decomposition
of the simplicial category $$S_{\ast}(\mathcal{M}^{qf, X\times
(\mathbb{A}^q - 0)}(X, \mathbb{A}^q) \rightarrow \mathcal{M}^{qf}(X,
\mathbb{A}^q)) = \mathcal{M}^+_{\ast} \times \mathcal{M}^-_{\ast},$$

where $$\mathcal{M}^+_{\ast} =
\mathcal{S}_{\ast}(\mathcal{M}^{fin}(X, \mathbb{A}^q-0) \rightarrow
\mathcal{M}^{fin}(X, \mathbb{A}^q))$$

and $\mathcal{M}^-_{n}$ has as its objects $(n + 1)$-filtered
bounded complexes
$$M^{\ast}_0 \subseteqq M^{\ast}_1 \subseteqq \dots \subseteqq
M^{\ast}_n$$ of coherent sheaves on $X\times \mathbb{A}^q$
satisfying the following conditions:

(1) $M^{j}_i$ is supported on a finite union of subschemes of $X
\times \mathbb{A}^q$ quasifinite over $X$ whose projection to $X$
does not contain the closed point of $X$.

(2) all the factors $M^{j}_i/ M^{j}_0$ are supported on subschemes
not intersecting $X\times 0$.

Now, $$w\mathcal{S}_{\ast}\mathcal{S}_{\ast}(\mathcal{M}^{qf,
X\times (\mathbb{A}^q - 0)}(X, \mathbb{A}^q) \rightarrow
\mathcal{M}^{qf}(X, \mathbb{A}^q)) =
w\mathcal{S}_{\ast}\mathcal{S}_{\ast}(\mathcal{M}^{fin}(X,
\mathbb{A}^q-0) \rightarrow \mathcal{M}^{fin}(X, \mathbb{A}^q))
\times w\mathcal{S}_{\ast}\mathcal{M}^-_{\ast}.$$

 Our objective is to show that
$w\mathcal{S}_{\ast}\mathcal{M}^-_{\ast}$ is contractible.

Notice that the condition (2) is void, since if $M^{j}_i/M^{j}_0$
was supported on a scheme $Z$ intersecting $X \times 0$ then the
projection of $Z$ to $X$ would contain the closed point of $X$.
Denote by $\mathcal{M}^{\sharp}$ the exact category whose objects
are complexes $M^{\ast}$ satisfying:

($\ast$)  $M^{i}$ are supported on subschemes of $X \times
\mathbb{A}^q$ that are quasifinite over $X$ and whose projections to
$X$ do not contain the closed point of $X$.

We have a trivial cofiber sequence:

$$|w\mathcal{S}_{\ast}\mathcal{M}^{\sharp}| \rightarrow
|w\mathcal{S}_{\ast}\mathcal{M}^{\sharp}| \rightarrow
|w\mathcal{S}_{\ast}\mathcal{M}^-_{\ast}|,$$ where the first map is
identity.  So the last spectrum is contractible and we showed that
the natural morphism of presheaves
$$\mathbf{K}(\mathcal{M}^{fin}(\ast, \mathbb{A}^q-0)
\rightarrow \mathcal{M}^{fin}(\ast, \mathbb{A}^q)) \rightarrow
\mathbf{K}(\mathcal{M}^{qf, X\times (\mathbb{A}^q - 0)}(\ast,
\mathbb{A}^q) \rightarrow \mathcal{M}^{qf}(\ast, \mathbb{A}^q))$$ is
a weak equivalence of presheaves.  Now, applying Lemma 2.16 we get
that
$$C_{\ast}\mathcal{K}(\mathcal{M}^{fin}(\ast, \mathbb{A}^q-0)
\rightarrow \mathcal{M}^{fin}(\ast, \mathbb{A}^q)) \rightarrow
C_{\ast}\mathcal{K}(\mathcal{M}^{qf, X\times (\mathbb{A}^q -
0)}(\ast, \mathbb{A}^q) \rightarrow \mathcal{M}^{qf}(\ast,
\mathbb{A}^q))$$ is a local weak equivalence of simplicial
presheaves.
\end{proof}
\subsection{Step 6}

Step $6$ concludes the proof of the comparison between the filtrations. This implies, by Nisnevich descent, that the associated spectral sequences, due respectively to Grayson and Friedlander-Sulin, are isomorphic. 
We now prove

\begin{proposition} The presheaf $C_{\ast}\mathcal{K}(\mathcal{P}^{qf,
\ast \times (\mathbb{A}^q - 0)}(\ast, \mathbb{A}^q)$ is weakly
contractible.
\end{proposition}
\begin{proof}

We establish the proposition by producing an explicit homotopy.  Precisely, by Lemma 2.5, we
need to construct a morphism of presheaves $$H:
\mathcal{K}(\mathcal{P}^{qf, \ast \times (\mathbb{A}^q - 0)}(\ast,
\mathbb{A}^q) \rightarrow C_1\mathcal{K}(\mathcal{P}^{qf, \ast
\times (\mathbb{A}^q - 0)}(\ast, \mathbb{A}^q)$$ such that
$i_0^{\ast}\circ H$ is homotopic to constant and $i_1^{\ast}\circ H$
is identity.  Let $X$ be a scheme.  Consider the map:  $$\theta: X
\times \mathbb{A}^q \times \mathbb{A}^1 \rightarrow X \times
\mathbb{A}^q$$ given by multiplication $$(x, v, t) \rightarrow (x,
tv).$$ It induces an exact functor

$$\theta^{\ast}: \mathcal{P}^{qf,X \times (\mathbb{A}^q-0)}(X,
\mathbb{A}^q) \rightarrow \mathcal{P}^{qf, X \times \mathbb{A}^1
\times (\mathbb{A}^q-0)}(X \times \mathbb{A}^1, \mathbb{A}^q)$$ that
in turn induces a map on $K$-theory spectra.  It is a
straightforward to check that it satisfies the required properties.
\end{proof}





%

\end{document}